\documentclass{amsart}

\usepackage{amsmath}
\usepackage{amssymb}
\usepackage[all,color]{xy}
\usepackage{subcaption}
\usepackage{amsthm}
\usepackage{color}
\usepackage{tikz}
\usetikzlibrary{positioning}
\usepackage{hyperref}
\usepackage[capitalize,nameinlink,noabbrev,nosort]{cleveref}
\usepackage{graphicx}

\crefname{lem}{Lemma}{Lemmas}

\newcommand{\taedge}[3]{
\raisebox{-3pt}{
\begin{tikzpicture}[node distance=0.8cm]
  \draw [->] (0,0) -- (1,0) 
  node[pos = 0, label = left:$#1$] {}
  node[midway, draw, blue, circle, fill = white, scale = 0.5, label=above:{$#3$}] {A}
  node[pos = 1, label = right:$#2$] {} ;
\end{tikzpicture}
}
}

\newcommand{\tbedge}[3]{
\raisebox{-3pt}{
\begin{tikzpicture}[node distance=0.8cm]
  \draw [->] (0,0) -- (1,0) 
  node[pos = 0, label = left:$#1$] {}
  node[midway, draw, red, circle, fill = white, scale = 0.5, label=above:{$#3$}] {B}
  node[pos = 1, label = right:$#2$] {} ;
\end{tikzpicture}
}
}

\newtheorem{thm}{Theorem}
\newtheorem{prop}[thm]{Proposition}
\newtheorem{cor}[thm]{Corollary}
\newtheorem{lem}[thm]{Lemma}

\newtheorem{defn}[thm]{Definition}

\newtheorem{rem}[thm]{Remark}

\DeclareMathOperator{\wgt}{wgt}
\DeclareMathOperator{\swgt}{swgt}
\DeclareMathOperator{\sgn}{sgn}

\begin{document}

\title[Combinatorial proofs of multivariate Cayley--Hamilton theorems]
{Combinatorial proofs of multivariate Cayley--Hamilton theorems}

\author{Arvind Ayyer}
\address{Arvind Ayyer, Department of Mathematics,
Indian Institute of Science, Bangalore  560012, India.}
\email{arvind@iisc.ac.in}

\author{Naren Sundaravaradan}
\address{Naren Sundaravaradan, HFN Inc., Bengaluru, India 560076}
\email{nano.naren@gmx.com}

\date{\today}

\begin{abstract}
We give combinatorial proofs of two multivariate Cayley--Hamilton type theorems. The first one is due to Phillips (Amer. J. Math., 1919) involving $2k$ matrices, of which $k$ commute pairwise.
The second one uses the mixed discriminant, a matrix function which has generated a lot of interest in recent times.
Recently, the Cayley--Hamilton theorem for mixed discriminants
was proved by Bapat and Roy (Comb. Math. and Comb. Comp., 2017).
We prove a Phillips-type generalization of the Bapat--Roy theorem, which involves $2nk$ matrices, where $n$ is the size of the matrices, among which $nk$ commute pairwise.
Our proofs generalize the univariate proof of Straubing (Disc. Math., 1983) for the original Cayley--Hamilton theorem in a nontrivial way, and involve decorated permutations and decorated paths.
\end{abstract}

\subjclass[2010]{05A19, 05A05, 05C20, 15A15}
\keywords{Cayley--Hamilton theorem, mixed discriminants, Phillips' theorem, combinatorial proof}

\maketitle

\section{Introduction}
\label{sec:intro}

Suppose $A$ is an $n \times n$ matrix with entries in a commutative ring. Then the Cayley--Hamilton theorem says that $p(A) = 0$, where $p(x) = \det(x I_n - A)$ is the characteristic polynomial of $A$ and $I_n$ is the $n \times n$ identity matrix.
The Cayley--Hamilton theorem is probably the first deep theorem one sees in linear algebra. It was first proved for linear functions of quaternions (corresponding to real $4 \times 4$ or complex $2 \times 2$ matrices) by Hamilton~\cite{hamilton-1853}. Cayley~\cite{cayley-1858} stated it for sizes $2$ and $3$, but gave a demonstration only in the former case.
Sylvester immediately realised its importance and popularized it, calling it the \emph{no-little-marvellous Hamilton--Cayley theorem}~\cite{sylvester-1884}.

The first proof was given by Buchheim~\cite{buchheim-1884} assuming invertibility of the matrix, but the first general proof was given by Frobenius~\cite{frobenius-1896}.  For more on the history of this remarkable theorem, see~\cite{crilly-1978}.
Several proofs are now known at various levels of abstraction\footnote{The \href{https://en.wikipedia.org/wiki/Cayley\%E2\%80\%93Hamilton_theorem}{Wikipedia article} on this topic itself gives four distinct proofs.}.
Relevant to this work is an elegant combinatorial proof due to Straubing~\cite{straubing-1983,zeilberger-1985}.

H. B. Phillips~\cite{phillips-1919} proved the following generalization of the Cayley--Hamilton theorem. Suppose $A_1,\dots,A_k$ and $B_1, \dots, B_k$ are two families of $n \times n$ matrices such that $B_i B_j = B_j B_i$ for all $1 \leq i < j \leq k$ and
\begin{equation}
\label{aibisum1}
    A_1B_1 + \dots + A_k B_k = 0.
\end{equation}

\begin{thm}[{\cite[Theorem I]{phillips-1919}}]
  \label{thm:phillips}
Define the polynomial $p(x_1, \dots, x_k) = \det(A_1x_1 \allowbreak + \dots + A_kx_k)$. Then $p(B_1, \dots, B_k) = 0$.
\end{thm}

We will give a combinatorial proof of \cref{thm:phillips} in \cref{sec:phillips}. 
For our proof, we will think of the entries in these matrices as formal commuting indeterminates.
An instructive special case about a pair of commuting matrices $A, B$ arises by setting $k = 2, A_1 = A, B_1 = B, A_2 = -B, B_2 = A$ as follows.

\begin{cor}
\label{cor:2matrices}
Let $A, B$ be commuting matrices. Define the bivariate polynomial $q(x,y) = \det(xA - yB)$. Then $q(B,A) = 0$.
\end{cor}

Setting $A$ equal to the identity matrix in \cref{cor:2matrices} reduces to the Cayley--Hamilton theorem.

We now move on to an important generalization of the determinant. For
an integer $n$, $[n] = \{1,\dots,n\}$ and $S_n$ stands for the set of
permutations $[n]$.

\begin{defn}
\label{def:mixeddisc}
The \emph{mixed discriminant} of an $n$-tuple $(A_1, \dots, A_n)$ of $n \times n$ matrices is defined as
\[
  \mathfrak{D}(A_1,\dots, A_n) = \frac{1}{n!} \sum_{\alpha \in S_n}
  \det \left( A_{\alpha_1}^{(1)} \; \big| \;\cdots\; \big| \; A_{\alpha_n}^{(n)} \right),
\]
where $A^{(i)}$ denotes the $i$'th column of the matrix $A$.
\end{defn}

The basic properties of the mixed discriminant are given in \cite{bapat-1989}. From the combinatorial point of view, it has been used to enumerate coloured spanning forests~\cite{bapat-constantine-1992}.
It simultaneously generalizes both the determinant and the permanent. For a fixed matrix $B$, $\mathfrak{D}(B,\dots,B) = \det(B)$, 
and if we set $B_i$ to be the diagonal matrix with entries $B_{i,1},\dots,B_{i,n}$, then $\mathfrak{D}(B_1,\dots,B_n) = \sum_{\sigma \in S_n} 
B_{1,\sigma_1} \cdots B_{n,\sigma_n}$, which is the permanent of $B$.

We will use $I$ for the identity matrix whenever the size is clear from the context.
Bapat and Roy~\cite{bapat-roy-2017} generalized the Cayley--Hamilton theorem for mixed discriminants by adapting Straubing's proof~\cite{straubing-1983}. 

\begin{thm}[{\cite[Theorem 1.1]{bapat-roy-2017}}]
\label{thm:cayley_discriminant}
For an $n$-tuple of $n \times n$ matrices $(A_1, \dots, A_n)$, define the polynomial
\[
  f(x_1,\dots, x_n) = \mathfrak{D}(x_1 I - A_1, \dots, x_n I - A_n).
\]
Then $f(A_1, \dots, A_n) = 0$.
\end{thm}

We note in passing that $f(x,\dots,x)$ is also known as the \emph{mixed characteristic polynomial} and was an important ingredient in the recent proof of the Kadison-Singer theorem~\cite{marcus-spielman-srivastava-2015}. 

For some positive integers $n, k$, let $(A_{i,j})_{i \in [n], j \in [k]}$ and $(B_{i,j})_{i \in [n], j \in [k]}$ be two families of $n \times n$ matrices, where $B_{i,j}B_{i',j'} = B_{i',j'}B_{i,j} \text{ for all } 1 \le i < i' \le n, 1 \le j, j' \le k$. In addition, suppose 
\begin{equation}
\label{aijbijsum1}
A_{i,1}B_{i,1} + \cdots + A_{i,k}B_{i,k} = 0, \quad 1 \le i \le n.
\end{equation}

\begin{thm}
\label{thm:mixeddisc}
For indeterminates $(x_{i,j})_{i \in [n], j \in [k]}$, define the polynomial
\begin{equation}
\label{defq}
 \hat{p} \left((x_{i,j})_{i \in [n], j \in [k]} \right) = \mathfrak{D}(A_{1,1}x_{1,1} + \dots + A_{1,k}x_{1,k}, \dots, A_{n,1}x_{n,1} + \dots + A_{n,k}x_{n,k}).
\end{equation}
Then
\[
\hat{p}\left((B_{i,j})_{i \in [n], j \in [k]} \right) = 0.
\]
\end{thm}

We will give a combinatorial proof of \cref{thm:mixeddisc} in \cref{sec:mixeddisc}. 
Even for this proof, we will think of the entries in these matrices as formal commuting indeterminates.
We now discuss a special case of \cref{thm:mixeddisc} for $k=2$. Let $M_1, \dots, M_n$ be a family of matrices. We then set
$A_{i,1} = -B_{i,2} = I$ and $A_{i,2} = B_{i,1} = M_i$ for $i \in [n]$. 
Then, this family of matrices automatically satisfies \eqref{aijbijsum1}.
For convenience, we will set $x_{i,1} = x_i$ and $x_{i,2} = y_i$. Then the polynomial in \eqref{defq} becomes
\[
\hat{p}_2 \left(x_1,\dots,x_n; y_1,\dots,y_n \right) = \mathfrak{D}(x_{1} I - M_{1} y_1, \dots, x_{n} I - M_{n} y_1).
\]

\begin{cor}
\label{cor:mixeddisc-k=2}
Suppose $M_1, \dots, M_n$ are a pairwise commuting family of matrices. Then
\[
\hat{p}_2 \left(M_1,\dots,M_n; -I_1,\dots,-I_n \right) = 0.
\]
\end{cor}

\cref{cor:mixeddisc-k=2} bears the same relation to \cref{thm:mixeddisc} as 
\cref{cor:2matrices} does to \cref{thm:phillips}. If we compare this result with \cref{thm:cayley_discriminant}, we see that the extra set of variables $y_i$ forces $M_i$'s to be pairwise commuting in order for the Cayley--Hamilton theorem to apply.

\begin{rem}
Suppose we choose matrices such that $A_{i,j} = A_j$ and $B_{i,j} = B_j$ as well as set variables $x_{i,j} = x_j$ for all $i$. Then \cref{thm:mixeddisc} reduces to \cref{thm:phillips}.
\end{rem}

The plan of the rest of the paper is as follows. We first give a combinatorial proof for \cref{thm:phillips} in \cref{sec:phillips}.
We will illustrate the key ideas of the proof using $2 \times 2$ matrices in \cref{sec:n=2}. We show how the proof relates to Straubing's proof of the Cayley--Hamilton theorem in \cref{sec:compare}. We also compare our proof to Phillips' original proof in \cref{sec:proof phillips}.
We then give a proof of \cref{thm:mixeddisc} in \cref{sec:mixeddisc}
using a naturally generalization of our proof strategy for \cref{thm:phillips}.
We illustrate the proof ideas again for $2 \times 2$ matrices in \cref{sec:mixed-n=2}.

\section{Proof of Phillips' theorem}
\label{sec:phillips}

Throughout this section, we will fix $k$ and $n \times n$ matrices $A_1, \dots, A_k$ and $B_1,\dots, \allowbreak B_k$ where $B_i$'s commute pairwise and the matrices satisfy 
\begin{equation}
\label{aibisum}
    A_1B_1 + \dots + A_k B_k = 0.
\end{equation}

We will first define the key combinatorial objects involved in the
proof.

\begin{defn}
Let $\pi = (\pi_1, \dots, \pi_n) \in S_n$. A \emph{decorated
  permutation} $\bar{\pi}$ of $\pi$ is an $n$-tuple of triples
$\bar{\pi}_i = (i, \pi_i, \ell_i)$ for $i \in [n]$, where each $\ell_i
\in [k]$ is called a \emph{label}.
We will denote $\bar{\pi}_i$ as \taedge{i}{\pi_i}{\ell_i} which has \emph{weight} $(A_{\ell_i})_{i,\pi_i}$.
The \emph{signed weight} of the decorated permutation is given by
\[
\swgt(\bar{\pi}) = \sgn(\pi) \prod_{i=1}^n (A_{\ell_i})_{i,\pi_i}.
\]
The set of all decorated permutations is denoted $\bar{S}_{n,k}$.
\end{defn}

Since there are $n!$ permutations and all labels are independently
chosen, the cardinality of $\bar{S}_{n,k}$ is $n! k^n$. Let
  $n=3$, $k=2$, and $\pi=(3,1,2)$. Then an example of a decorated
  permutation is
\begin{equation}\label{eq:decorated_permutation}
  \bar{\pi}:\quad \taedge{1}{3}{1} \taedge{2}{1}{2} \taedge{3}{2}{2}
\end{equation}
with $\swgt(\bar{\pi}) = +(A_1)_{1,3}(A_2)_{2,1}(A_2)_{3,2}$.

\begin{defn}
A \emph{decorated path} of length $n$ is a tuple $\bar{q} = (q_1,\dots,q_{n+1})$,
where each $q_i \in [n]$. For $i \in [n]$, the $i$'th labeled edge is denoted $\bar{q}_i = \tbedge{q_i}{ q_{i+1}}{ \ell_i}$
and has weight $(B_{\ell_i})_{q_i,q_{i+1}}$, where the \emph{label} $\ell_i \in [k]$.
The \emph{weight} of the decorated path is
\[
\wgt(\bar{q}) = \prod_{i=1}^n (B_{\ell_i})_{q_i,q_{i+1}}.
\]
The set of all decorated paths is denoted $\overline{Q}_{n,k}$.
\end{defn}
For instance, with $n=3$ and $k=2$,
\begin{equation}\label{eq:decorated_path}
  \bar{q}:\quad \tbedge{3}{1}{1} \tbedge{1}{2}{2} \tbedge{2}{1}{2}
\end{equation}
is a decorated path with $\wgt(\bar{q}) = (B_1)_{3,1}(B_2)_{1,2}(B_2)_{2,1}$.

\begin{defn}
A \emph{pathmutation} is a pair $(\bar{\pi}, \bar{q})$ where $\bar{\pi} \in \bar{S}_{n,k}, \bar{q} \in \overline{Q}_{n,k}$ such that the labels of the $i$'th element of the permutation and the $i$'th edge of the path are the same for all $i \in [n]$.
The \emph{signed weight} of a pathmutation is 
\[
\wgt(\bar{\pi}, \bar{q}) = \swgt(\bar{\pi}) \wgt(\bar{q}).
\]
The set of pathmutations beginning with $q_1 = b$ and ending with $q_{n+1} = e$ is denoted $\mathcal{A}(b,e)$.
\end{defn}

The cardinality of $\mathcal{A}(b,e)$ is $n! k^n n^{n-1}$ for every $b,e \in [n]$ because we can choose $q_2,\dots,q_{n-1}$ arbitrarily.

\begin{center}
  \begin{figure}[h!]
    \begin{tabular}{c c c c c c}
      $\bar{\pi}$: & \taedge{1}{\pi_1}{\ell_1} & $\cdots$ & \taedge{s}{\pi_s}{\ell_s} & $\cdots$ & \taedge{n}{\pi_n}{\ell_n}\\
      $\bar{q}$: & \tbedge{q_1}{q_2}{\ell_1} & $\cdots$ & \tbedge{q_s}{q_{s+1}}{\ell_s} & $\cdots$ & \tbedge{q_n}{q_{n+1}}{\ell_n}
    \end{tabular}
    \caption{An illustration of a generic pathmutation $(\bar{\pi}, \bar{q})$.}
    \label{fig:general_PQ}
  \end{figure}
\end{center}

See \cref{fig:general_PQ} for a generic pathmutation. We then set
\[
\swgt(\mathcal{A}(b,e)) = \sum_{(\bar{\pi}, \bar{q}) \in \mathcal{A}(b,e)} \swgt(\bar{\pi}) \wgt(\bar{q}).
\]
We will need more general objects than decorated permutations in our proofs, which we now define.

\begin{defn}
A \emph{decorated map} $\overline{m}$ is an $n$-tuple of triples
$\overline{m}_i = (\sigma_i, \tau_i, \ell_i)$, 
where $\sigma = (\sigma_1,\dots,\sigma_n) \in S_n$ is either the identity or a single transposition, $\tau_i \in [n]$ and $\ell_i \in [k]$ for all $i$ such that
\begin{itemize}
\item $\# \{\tau_1,\dots, \tau_n \} \geq n-1$,
\item if $\# \{\tau_1,\dots, \tau_n \} = n$, then $\sigma$ is the identity permutation,
\item if $\tau_i = \tau_j$ for some $(i,j)$, then either 
$\sigma_i = i, \sigma_j = j$ or $\sigma_i = j, \sigma_j = i$.
\end{itemize}
The \emph{weight} of $\overline{m}_i$ is $(A_{\ell_i})_{\sigma_i,\tau_i}$ and is denoted \taedge{\sigma_i}{\tau_i}{\ell_i}. The weight of $\overline{m}$ is then
\[
\wgt(\overline{m}) = \prod_{i=1}^n (A_{\ell_i})_{\sigma_i,\tau_i}.
\]
The set of all decorated maps is denoted $\overline{M}_{n,k}$.
\end{defn}

When $\# \{\tau_1,\dots, \tau_n \} = n$, we get exactly decorated permutations. When $\# \{\tau_1,\allowbreak  \dots, \tau_n \} = n-1$, there are $n(n-1) \times n!/2$ possibilities for $\tau$ and $2$ possibilities for $\sigma$ in each case so that the cardinality of $\overline{M}_{n,k}$ is
\[
k^n (n! + n(n-1) n!) = n! k^n (n^2 - n + 1).
\]
For example, we can view \eqref{eq:decorated_permutation} as the
  decorated map
  \[
    \overline{m}: ((1,3,1), (2,1,2), (3,2,2)),
  \]
  where the first component $\sigma$ is the identity permutation and
  the second component $\tau$ is the permutation $\pi=(3,1,2)$. Now
  suppose we fix $\tau_2=1$, $\tau_1 = \tau_3$, and the same labels as
  above. Then $\sigma$ is forced to be either $(1,2,3)$ (i.e. the
  identity) or $(3,2,1)$, and the four possible decorated maps are
\begin{align}
  &\taedge{1}{3}{1} \taedge{2}{1}{2} \taedge{3}{3}{2},\label{eq:decorated_map_1}\\
  &\taedge{3}{3}{1} \taedge{2}{1}{2} \taedge{1}{3}{2},\label{eq:decorated_map_2}\\
  &\taedge{1}{2}{1} \taedge{2}{1}{2} \taedge{3}{2}{2},\label{eq:decorated_map_3}\\
  &\taedge{3}{2}{1} \taedge{2}{1}{2} \taedge{1}{2}{2}.\label{eq:decorated_map_4}
\end{align}

\begin{defn}
\label{def:pathmap}
A \emph{pathmap} is a pair $(\overline{m},\bar{q})$ where $\overline{m} \in \overline{M}_{n,k}, \bar{q} \in 
\overline{Q}_{n,k}$ such that 

\begin{itemize}

\item If $\{\tau_1,\dots,\tau_n\} = [n]$, then $(\overline{m},\bar{q})$ is a pathmutation.

\item If $\# \{\tau_1,\dots, \tau_n \} = n-1$ and $\tau_s = \tau_t$ for some $s < t$, then $q_1 = \tau_s = \tau_t$. In this case, the labels of $\overline{m}_k$ and $\bar{q}_k$ must match for all $k \neq s,t$. In addition, if $\sigma_s = s, \sigma_t = t$ (resp. $\sigma_s = t, \sigma_t = s$), then the labels of $\overline{m}_s$ and $\overline{m}_t$ are equal to those of $\bar{q}_s, \bar{q}_t$ (resp. $\bar{q}_t, \bar{q}_s$) respectively.

\end{itemize}

The weight of the pathmap $(\overline{m},\bar{q})$ is
\[
\wgt(\overline{m},\bar{q}) = \wgt(\overline{m}) \wgt(\bar{q}).
\]

The set of pathmaps with $\# \{\tau_1,\dots, \tau_n \} = n-1$ such that $\{\tau_1,\dots, \tau_n \} = [n] \setminus \{b\}$ and ending with $q_{n+1} = e$ is denoted $\mathcal{H}(b,e)$. 
In addition, let $\mathcal{G}(b,e) = \mathcal{H}(b,e) \cup \mathcal{A}(b,e)$.
\end{defn}
For instance, we may combine the decorated permutation \eqref{eq:decorated_permutation} and the
decorated path \eqref{eq:decorated_path} to get a pathmutation in $\mathcal{A}(3, 1)$:
\begin{center}
  \begin{tabular}{c c c c}
    $\bar{\pi}$: & \taedge{1}{3}{1} & \taedge{2}{1}{2} & \taedge{3}{2}{2}\\
    $\bar{q}$: & \tbedge{3}{1}{1} & \tbedge{1}{2}{2} & \tbedge{2}{1}{2}
  \end{tabular}
\end{center}
where the labels match. We can also combine the decorated maps in
\eqref{eq:decorated_map_1} and the same decorated path $\bar{q}$ to
get the pathmap
\begin{center}
  \begin{tabular}{c c c c}
    $\bar{\pi}$: & \taedge{1}{3}{1} & \taedge{2}{1}{2} & \taedge{3}{3}{2}\\
    $\bar{q}$: & \tbedge{3}{1}{1} & \tbedge{1}{2}{2} & \tbedge{2}{1}{2}
  \end{tabular}
\end{center}
However, the combination of the decorated map
\eqref{eq:decorated_map_2} with $\bar{q}$ is not a pathmap because the
condition on the labels is not satisfied. Further, the decorated map
\eqref{eq:decorated_map_3} with $\bar{q}$ does not form a pathmap
because $3 = q_1 \ne \tau_1 = \tau_3 =2$. Lastly,
\eqref{eq:decorated_map_4} with $\bar{q}$ fails both conditions.

In other words $\mathcal{G}(b,e)$ consists of two kinds of elements $(\overline{m},\bar{q})$. Those with $q_1 = b$ are pathmutations and the remaining are
elements of $\mathcal{H}(b,e)$, which we count now. 
For every fixed $b$ and $e$, there are $n-1$ possibilities for $q_1$, $n$ possibilities each for $q_1,\dots,q_n$, $k$ possibilities each for $\ell_1, \dots, \ell_n$,  $n!/2$ arrangements of $\tau$ and $2$ arrangements for $\sigma$. Therefore, $\#\mathcal{H}(b,e) = (n-1) n! k^n n^{n-1}$ and the cardinality of $\mathcal{H}(b,e)$ is $n-1$ times that of $\mathcal{A}(b,e)$.

\cref{fig:Gij} illustrates the two kinds of elements in $\mathcal{H}(b,e)$  in the second condition in \cref{def:pathmap}.

\begin{center}
  \begin{figure}[h!]
      \begin{tabular}{c c c c c c}
		$\overline{m}$: & $\cdots$ & \taedge{s}{q_1}{\ell_s} & $\cdots$ & \taedge{t}{q_1}{\ell_{t}} & $\cdots$ \\
        $\bar{q}$: & $\cdots$ & \tbedge{q_s}{q_{s+1}}{\ell_s} & $\cdots$ & \tbedge{q_{t}}{q_{t+1}}{\ell_{t}} & $\cdots$ 
      \end{tabular}
      
      (a) $(\overline{m},\bar{q})$

\medskip      
      \begin{tabular}{c c c c c c}
        $\overline{m}$: & $\cdots$ & \taedge{t}{q_1}{\ell_{t}} & $\cdots$ & \taedge{s}{q_1}{\ell_s} & $\cdots$ \\
        $\bar{q}$: & $\cdots$ & \tbedge{q_s}{q_{s+1}}{\ell_s} & $\cdots$ & \tbedge{q_{t}}{q_{t+1}}{\ell_{t}} & $\cdots$ 
      \end{tabular}

      (b) $(\overline{m}',\bar{q})$

\medskip  
    \caption{Two elements $(\overline{m},\bar{q})$ and $(\overline{m}',\bar{q})$ of $\mathcal{H}(b, e)$ such that $\overline{m}_k = \overline{m}'_k$ for $k \neq s,t$. Note that $\overline{m}'_s = \overline{m}_t$ and $\overline{m}'_t = \overline{m}_s$. If we write $\overline{m}_i = (\sigma_i, \tau_i, \ell_i)$ and $\overline{m}'_i = (\sigma'_i, \tau'_i, \ell_i)$, then
$\sigma$ is the identity permutation, $\sigma'$ is the transposition $(s,t)$, and $\tau_s = \tau_t = \tau'_s = \tau'_t = q_1$. }
\label{fig:Gij}
  \end{figure}
\end{center}

To assign a sign to the elements of $\mathcal{G}(b,e)$, we define a map $\phi : \mathcal{A}(b,e) \times [n] \rightarrow \mathcal{G}(b,e)$ defined by $\phi((\bar{\pi},\bar{q}),j) = (\overline{m}',\bar{q}')$ as follows.
First, define $\bar{q}'$ by
\[
\bar{q}'_r = 
\begin{cases}
\tbedge{j}{q_2}{\ell_1} & r = 1, \\
\bar{q}_r & \text{otherwise}.
\end{cases}
\]
Next, set $s = \pi^{-1}_b$ and $t = \pi^{-1}_j$. 
Then let
\begin{equation}
\label{defphi}
\overline{m}'_r = 
\begin{cases}
\taedge{{\min(s,t)}}{j}{\ell_{\min(s,t)}} & \text{if $r = s$}, \\
\taedge{{\max(s,t)}}{j}{\ell_{\max(s,t)}} & \text{if $r = t$},\\
\bar{\pi}_r & \text{otherwise}.
\end{cases}
\end{equation}

\begin{prop}
\label{prop:bij}
$\phi$ is a bijection.
\end{prop}

\begin{proof}
We prove this by constructing the inverse map.
  Let $(\overline{m}',\bar{q}') \in \mathcal{G}(b,e)$ and $\overline{m}'_i =
  (\sigma'_i, \tau'_i, \ell'_i)$, $i \in [n]$. 
  \[
    \bar{q}_r =
    \begin{cases}
      \tbedge{b}{ q'_2}{ \ell_1} &\text{if } r = 1\\
      \bar{q}'_r &\text{otherwise}.
    \end{cases}
  \]
  If $q'_1 = b$, then set $\bar{\pi}_r = \overline{m}'_r$; otherwise, there exists $1 \le s < s' \le n$ such that $\tau'_s = \tau'_{s'} = q'_1$. In this case, set
  \[
    \bar{\pi}_r =
    \begin{cases}
      \taedge{s}{b}{\ell_s} & \text{if } r = s, \sigma'_r = s,\\
      \taedge{s}{q'_1}{\ell_s} & \text{if } r = s, \sigma'_r = s',\\
      \taedge{s'}{b}{\ell_{s'}} & \text{if } r = s', \sigma'_r = s,\\
      \taedge{s'}{q'_1}{\ell_{s'}} & \text{if } r = s', \sigma'_r = s',\\
      \overline{m}'_r &\text{otherwise}.
    \end{cases}
  \]
  Clearly, $(\bar{\pi},\bar{q}) \in \mathcal{A}(b,e)$. 
  It is routine to check that $\phi((\bar{\pi},\bar{q}),q'_1) = (\overline{m}',\bar{q}')$.
\end{proof}

Note also that $\phi((\bar{\pi},\bar{q}),b) = (\bar{\pi},\bar{q})$ for $(\bar{\pi},\bar{q}) \in \mathcal{A}(b,e)$. 
We now use \cref{prop:bij} to give a signed weight to a pathmap $(\overline{m}',\bar{q}')$.
Suppose $\phi^{-1}(\overline{m}',\bar{q}') = ((\bar{\pi},\bar{q}),k)$. Then set
\begin{equation}
\label{wt-pathmap}
\swgt(\overline{m}',\bar{q}') = \sgn(\pi) \wgt(\overline{m}',\bar{q}').
\end{equation}

\begin{lem}
\label{lem:Gij}
Let $A_1,\dots,A_k$, $B_1,\dots,B_k$ be $n \times n$ matrices satisfying \eqref{aibisum} and where the $B_i$'s commute pairwise, and let $b,e \in [n]$. Then
\[
\sum_{(\overline{m},\bar{q}) \in \mathcal{G}(b,e)} \swgt(\overline{m},\bar{q}) = 0.
\]
\end{lem}

\begin{proof}
By definition of \eqref{wt-pathmap},
\[
\swgt(\mathcal{G}(b,e)) = \sum_{(\bar{\pi},\bar{q}) \in \mathcal{A}(b,e)}
\sum_{a=1}^n \swgt \left( \phi((\bar{\pi},\bar{q}), a) \right).
\]
We will refine the sum according to the underlying permutation $\pi$ and all the labels except $\ell_s$, where $\pi_s = b$.
Thus,
\[
\swgt(\mathcal{G}(b,e)) = \sum_{\pi \in S_n}
\sum_{\substack{(\bar{\pi},\bar{q}) \in \mathcal{A}(b,e) \\ 
1 \leq \ell_1,\dots,\ell_{s-1}, \ell_{s+1}, \dots,\ell_n \leq k}}
\sum_{1 \le q_2, \dots, q_n \le n}
\sum_{\ell_s = 1}^k \sum_{a=1}^n \swgt \left( \phi((\bar{\pi},\bar{q}), a) \right).
\]
We will now perform the three inner sums. 
The common factor for these sums is 
\[
\sgn(\pi) \prod_{\substack{i = 1 \\ i \neq s}}^n \wgt ( \taedge{i}{\pi_i}{\ell_i}) =
\sgn(\pi) \prod_{\substack{i = 1 \\ i \neq s}}^n (A_{\ell_i})_{i,\pi_i}.
\]
Since all three are independent, we can perform them in any order. We first perform
\begin{equation}
\begin{split}
\label{eq:lemma1_first}
\sum_{\ell_s = 1}^k \sum_{a=1}^n & \wgt(\taedge{s}{a}{\ell_s}) \\
& \times \sum_{1 \le q_2, \dots, q_n \le n} 
\wgt\left( \tbedge{a}{q_2}{\ell_1} \right) 
\cdots 
\wgt \left( \tbedge{q_n}{e}{\ell_n} \right).
\end{split}
\end{equation}
Using the pairwise commutativity of $B_1, \dots, B_k$, cycle the labels $\ell_1, \dots, \ell_s$ in the path to bring $\ell_s$ to the front so that we have
\begin{multline*}
\sum_{\ell_s=1}^k \sum_{a = 1}^n \wgt(\taedge{s}{a}{\ell_s}) \sum_{1 \le q_2, \dots, q_n \le n} 
\wgt\left( \tbedge{a}{q_2}{\ell_s} \right) \\
\times \wgt \left( \tbedge{q_2}{q_3}{\ell_1} \right) \cdots 
\wgt \left(\tbedge{q_s}{q_{s+1}}{\ell_{s-1}} \right) \\ 
\times \wgt\left( \tbedge{q_{s+1}}{q_{s+2}}{\ell_{s+1}} \right)
\cdots \wgt \left( \tbedge{q_n}{e}{\ell_n} \right).
\end{multline*}
We now perform the sum over $a$ and $\ell_s$ first. This amounts to 
\begin{equation}
\label{eq:lemma1_second}
\sum_{\ell_s=1}^k \sum_{a = 1}^n \wgt(\taedge{s}{a}{\ell_s}) \wgt( \tbedge{a}{q_2}{\ell_s})
= \sum_{\ell_s=1}^k \sum_{a = 1}^n (A_{\ell_s})_{s,a} (B_{\ell_s})_{a,q_2},
\end{equation}
which, by matrix multiplication is the $(s,q_2)$'th entry of $A_1B_1 + \dots + A_kB_k$, which is zero by \eqref{aibisum}. This completes the proof.
\end{proof}

\begin{lem}
Let $A_1,\dots,A_k$, $B_1,\dots,B_k$ be $n \times n$ matrices satisfying \eqref{aibisum} and where the $B_i$'s commute pairwise, and let $b,e \in [n]$. Then
\label{lem:Hij}
\[
\sum_{(\overline{m},\bar{q}) \in \mathcal{H}(b,e)} \swgt(\overline{m},\bar{q}) = 0.
\]
\end{lem}

\begin{proof}
By \cref{prop:bij}, every pair in $\mathcal{H}(b,e)$ is equal to 
$\phi((\bar{\pi},\bar{q}),j)$ for some $(\bar{\pi},\bar{q}) \in \mathcal{A}(b,e)$ and $1 \leq j \leq n$, $j \neq b$. 
Define a map $f : \mathcal{H}(b,e) \rightarrow \mathcal{H}(b,e)$ such that
if $f(\phi((\bar{\pi},\bar{q}), j)) = \phi((\bar{\pi}',\bar{q}), j)$, then
\begin{align*}
\bar{\pi}'_r = \begin{cases}
\taedge{r}{ j}{ \ell_r} &\text{if } \pi_r = b,\\
\taedge{r}{ b}{ \ell_{r}} &\text{if } \pi_r = j,\\
\bar{\pi}_r &\text{otherwise}.
\end{cases}
\end{align*}
Clearly, $f$ is an involution. We claim that it is sign-reversing and weight-preserving.
Let $(\overline{m}, \bar{q}) = \phi((\bar{\pi},\bar{q}),j)$ and suppose that $s = \pi^{-1}_b < t = \pi^{-1}_j$. 
Then, by \eqref{defphi}, we have 
\begin{align*}
&\bar{\pi}_s = \taedge{s}{b}{\ell_s}, \quad \bar{\pi}_t = \taedge{t}{j}{\ell_{t}},\\
&\overline{m}_s = \taedge{s}{j}{\ell_s}, \quad \overline{m}_t =  \taedge{t}{j}{\ell_{t}}.
\end{align*}
By the definition of $f$, $\bar{\pi}'$ will have
\[
\bar{\pi}'_s = \taedge{s}{j}{\ell_{s}}, \quad \bar{\pi}'_{t} = \taedge{t}{b}{\ell_{t}}.
\]
Let $(\overline{m}',\bar{q}) = \phi((\bar{\pi}',\bar{q}), j)$, then 
\[
\overline{m}'_s : \taedge{t}{ j}{ \ell_{t}}, \quad \overline{m}'_{t} : \taedge{s}{j}{\ell_s}.
\]
Thus, the weights of $\overline{m}$ and $\overline{m}'$ are the same, and $\pi$ and $\pi'$ differ by a single transposition. 
Hence, $\swgt(\overline{m}',\bar{q}) = -\swgt(\overline{m},\bar{q})$ by \eqref{wt-pathmap}.
The case of $s > t$ proceeds in a very similar manner. 
\end{proof}

\begin{proof}[Proof of \cref{thm:phillips}]
We first claim that
\[
  \swgt(\mathcal{A}(b,e)) = p(B_1, \dots, B_k)_{b,e}.
\]
To see this, begin by expanding the polynomial $p$ as
\[
p(x_1, \dots, x_k) = \sum_{\sigma \in S_n} \sgn(\sigma) \prod_{r=1}^n \left( (A_1)_{r, \sigma_r} x_1 + \cdots + (A_k)_{r, \sigma_r} x_k \right).
\]
Now, substitute $x_i$ by $B_i$ and use the fact that $B_i$'s commute pairwise to obtain
\begin{align*}
p(B_1, \dots, B_k) =&  \sum_{\sigma \in S_n} \sgn(\sigma) \prod_{r=1}^n \left( (A_1)_{r, \sigma_r}B_1 + \dots + (A_k)_{r, \sigma_r}B_k \right) \\
=&  \sum_{\sigma \in S_n} \sgn(\sigma) 
\prod_{(z_1, \dots, z_n) \in [k]^n} (A_{z_1})_{1, \sigma_1} \cdots 
(A_{z_n})_{n, \sigma_n} B_{z_1} \cdots B_{z_n}.
\end{align*}
Now consider the $(b,e)$'th entry of this sum.
For each permutation $\sigma$ and each element $z = (z_1, \dots, z_n) \in [k]^n$, we obtain a decorated permutation $\bar{\sigma}$, the label of whose $i$'th element is $z_i$ as seen above.
Now expand the product of $B_{z_i}$'s on the right hand side.
The $(b,e)$'th entry is a sum of terms, each of which corresponds exactly to a decorated path with initial vertex $b$ and final vertex $e$.
This proves the claim above.

Now, we have by construction, $\mathcal{G}(b,e) = \mathcal{A}(b,e) \cup \mathcal{H}(b,e)$.
We have proved that $\swgt(\mathcal{G}(b,e)) = 0$ in \cref{lem:Gij} and that $\swgt(\mathcal{H}(b,e)) = 0$ in \cref{lem:Hij}.
Therefore, we have shown $\swgt(\mathcal{A}(b,e)) = 0$, completing the proof.
\end{proof}

\subsection{Illustration for $n=2$}
\label{sec:n=2}

The essence of the proof of \cref{thm:phillips} is contained in \cref{lem:Gij,lem:Hij}. We illustrate the ideas behind the proofs of these lemmas by looking at the case of $n=k=2$ in detail for $b = 1$ and $e = 2$. 
We will keep the labels $\ell_1 = \alpha$ and $\ell_2 = \beta$ arbitrary, so that we have $4$ pathmutations, which are shown in the left columns of \cref{fig:eg-n=2ab,fig:eg-n=2cd,fig:eg-n=2efgh}. 
Similarly, there are $(2-1) 2! 2^1 = 4$ such pathmaps in $\mathcal{H}(1,2)$, which are shown in the right columns of \cref{fig:eg-n=2ab,fig:eg-n=2cd,fig:eg-n=2efgh}. 

\begin{center}
\begin{figure}[h!]
\begin{tabular}{c c}
      \resizebox{0.45\textwidth}{!}{
    \begin{tabular}{c c c}
        $\bar{\pi}$: & \taedge{1}{1}{\alpha} & \taedge{2}{2}{\beta}\\[0.2cm]
        $\bar{q}$: & \tbedge{1}{1}{\alpha} & \tbedge{1}{2}{\beta}
    \end{tabular}}
&
      \resizebox{0.45\textwidth}{!}{
    \begin{tabular}{c c c}
        $\overline{m}$: & \taedge{1}{{2}}{\alpha} & \taedge{2}{2}{\beta}\\[0.2cm]
        $\bar{q}$: & \tbedge{{2}}{1}{\alpha} & \tbedge{1}{2}{\beta}
    \end{tabular}}
\\
    (a) {\tiny $+(A_\alpha)_{1,1}(A_\beta)_{2,2}(B_\alpha)_{1,1}(B_\beta)_{1,2}$}
&
    (b) {\tiny $+(A_\alpha)_{1,2}(A_\beta)_{2,2}(B_\alpha)_{2,1}(B_\beta)_{1,2}$}
\end{tabular}
\caption{The terms proportional to $(A_\beta)_{2,2}(B_\beta)_{1,2}$ along with their signed weights.}
\label{fig:eg-n=2ab}
\end{figure}
\end{center}

\begin{center}
\begin{figure}[h!]
\begin{tabular}{c c}
      \resizebox{0.45\textwidth}{!}{
    \begin{tabular}{c c c}
        $\bar{\pi}$: & \taedge{1}{1}{\alpha} & \taedge{2}{2}{\beta}\\[0.2cm]
        $\bar{q}$: & \tbedge{1}{2}{\alpha} & \tbedge{2}{2}{\beta}
    \end{tabular}}
&
      \resizebox{0.45\textwidth}{!}{
    \begin{tabular}{c c c}
        $\overline{m}$: & \taedge{1}{{2}}{\alpha} & \taedge{2}{2}{\beta}\\[0.2cm]
        $\bar{q}$: & \tbedge{{2}}{2}{\alpha} & \tbedge{2}{2}{\beta}
    \end{tabular}}
\\
    (c) {\tiny $+(A_\alpha)_{1,1}(A_\beta)_{2,2}(B_\alpha)_{1,2}(B_\beta)_{2,2}$}
&
    (d) {\tiny $+(A_\alpha)_{1,2}(A_\beta)_{2,2}(B_\alpha)_{2,2}(B_\beta)_{2,2}$}
\end{tabular}
\caption{The terms proportional to $(A_\beta)_{2,2}(B_\beta)_{2,2}$ along with their signed weights.}
\label{fig:eg-n=2cd}
\end{figure}
\end{center}

We now illustrate \cref{lem:Gij} for $s=1$. This will amount to
  summing over all configurations in \cref{fig:eg-n=2ab,fig:eg-n=2cd}. First compare the pathmutation $(\bar\pi,\bar{q})$ in \cref{fig:eg-n=2ab}(a) and the pathmap  $(\overline{m}',\bar{q}')$ in \cref{fig:eg-n=2ab}(b).
To explain the sign of the pathmap, note that
$\phi^{-1}(\overline{m}',\bar{q}')$ is given by 
\[
\begin{array}{ccc}
\bar{\pi}: &       \taedge{1}{1}{\alpha} & \taedge{2}{2}{\beta}\\
\bar{q}: &        \tbedge{1}{1}{\alpha} & \tbedge{1}{2}{\beta}
\end{array},
\]
using \cref{prop:bij}. Thus the corresponding permutation according to \eqref{wt-pathmap} is $(1,2)$. Now, the sum of weights of these are
\begin{align*}
& \sum_{r=1}^2 \sum_{s=1}^2 (A_\beta)_{2,2}(B_\beta)_{1,2} \Big(
(A_\alpha)_{1,1}(B_\alpha)_{1,1} + (A_\alpha)_{1,2}(B_\alpha)_{2,1} \Big) \\
=& \sum_{s=1}^2 (A_\beta)_{2,2}(B_\beta)_{1,2} \sum_{r=1}^2 (A_\alpha B_\alpha)_{1,1},
\end{align*}
which is zero by \eqref{aibisum}. A very similar computation goes through for 
the terms in \cref{fig:eg-n=2cd}(c) and (d).

\begin{center}
\begin{figure}[h!]
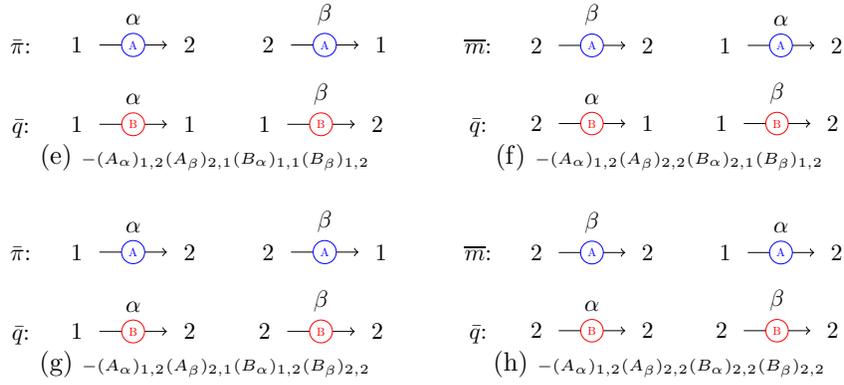

\begin{tabular}{c c}
      \resizebox{0.45\textwidth}{!}{
    \begin{tabular}{c c c}
        $\bar{\pi}$: & \taedge{1}{2}{\alpha} & \taedge{2}{1}{\beta}\\[0.2cm]
        $\bar{q}$: & \tbedge{1}{1}{\alpha} & \tbedge{1}{2}{\beta}
    \end{tabular}}
&
      \resizebox{0.45\textwidth}{!}{
    \begin{tabular}{c c c}
        $\overline{m}$: & \taedge{2}{{2}}{\beta} & \taedge{1}{2}{\alpha}\\[0.2cm]
        $\bar{q}$: & \tbedge{{2}}{1}{\alpha} & \tbedge{1}{2}{\beta}
    \end{tabular}}
\\
    (e) {\tiny $-(A_\alpha)_{1,2}(A_\beta)_{2,1}(B_\alpha)_{1,1}(B_\beta)_{1,2}$}
&
    (f) {\tiny $-(A_\alpha)_{1,2}(A_\beta)_{2,2}(B_\alpha)_{2,1}(B_\beta)_{1,2}$}
\\[0.4cm]
      \resizebox{0.45\textwidth}{!}{
    \begin{tabular}{c c c}
        $\bar{\pi}$: & \taedge{1}{2}{\alpha} & \taedge{2}{1}{\beta}\\[0.2cm]
        $\bar{q}$: & \tbedge{1}{2}{\alpha} & \tbedge{2}{2}{\beta}
    \end{tabular}}
&
      \resizebox{0.45\textwidth}{!}{
    \begin{tabular}{c c c}
        $\overline{m}$: & \taedge{2}{{2}}{\beta} & \taedge{1}{2}{\alpha}\\[0.2cm]
        $\bar{q}$: & \tbedge{{2}}{2}{\alpha} & \tbedge{2}{2}{\beta}
    \end{tabular}}
\\
    (g) {\tiny $-(A_\alpha)_{1,2}(A_\beta)_{2,1}(B_\alpha)_{1,2}(B_\beta)_{2,2}$}
&
    (h) {\tiny $-(A_\alpha)_{1,2}(A_\beta)_{2,2}(B_\alpha)_{2,2}(B_\beta)_{2,2}$}
\end{tabular}

  \caption{The terms proportional to $(A_\alpha)_{1,2}$ along with their signed weights.}
  \label{fig:eg-n=2efgh}
\end{figure}
\end{center}

We now illustrate \cref{lem:Gij} for $s=2$. This will amount to
  summing over all possible configurations in \cref{fig:eg-n=2efgh}. Complications arise in the remaining terms shown in \cref{fig:eg-n=2efgh}(e), (f), (g) and (h). The sign for the terms in (f) and (h) are computed as described above. In this case, combining terms (e) and (g), we get
\begin{align*}
& -\sum_{r=1}^2 \sum_{s=1}^2 (A_\alpha)_{1,2}(A_\beta)_{2,1} \Big(
(B_\alpha)_{1,1}(B_\beta)_{1,2} + (B_\alpha)_{1,2}(B_\beta)_{2,2} \Big) \\
=& -\sum_{r=1}^2 \sum_{s=1}^2 (A_\alpha)_{1,2}(A_\beta)_{2,1} \sum_{r=1}^2 (B_\alpha 
B_\beta)_{1,2} \\
=& -\sum_{r=1}^2 \sum_{s=1}^2 (A_\alpha)_{1,2}(A_\beta)_{2,1} \sum_{r=1}^2 (B_\beta B_\alpha)_{1,2} \\
=& -\sum_{r=1}^2 \sum_{s=1}^2 (A_\alpha)_{1,2}(A_\beta)_{2,1} \Big(
(B_\beta)_{1,1}(B_\alpha)_{1,2} + (B_\beta)_{1,2}(B_\alpha)_{2,2} \Big),
\end{align*}
where we have used the commutativity of $B_\alpha$ and $B_\beta$ in the third line. Similarly, combining terms (f) and (h), we get
\begin{align*}
-\sum_{r=1}^2 \sum_{s=1}^2 (A_\alpha)_{1,2}(A_\beta)_{2,2} \Big(
(B_\beta)_{2,1}(B_\alpha)_{1,2} + (B_\beta)_{2,2}(B_\alpha)_{2,2} \Big).
\end{align*}
Now, add the first summands in both the above equations to obtain
\begin{align*}
&  -\sum_{r=1}^2 \sum_{s=1}^2 (A_\alpha)_{1,2} (B_\alpha)_{1,2} \Big(
(A_\beta)_{2,1} (B_\beta)_{1,1} + (A_\beta)_{2,2}(B_\beta)_{2,1} \Big) \\
=& \sum_{r=1}^2(A_\alpha)_{1,2} (B_\alpha)_{1,2} \sum_{s=1}^2 (A_\beta B_\beta)_{2,1},
\end{align*}
which is now 0 by \eqref{aibisum}. A similar computation goes through for the sums involving the second and fourth summands. This computation is what is essentially carried out in \cref{lem:Gij}.

Now focus on the pathmap terms, namely (b), (d), (f) and (h). The (b) and (f) terms have the same weights but opposite signs. Ditto for (d) and (h) terms. This is an illustration of the sign-reversing involution in the proof of \cref{lem:Hij}.

\subsection{Reduction to the Cayley--Hamilton theorem}
\label{sec:compare}

The Cayley--Hamilton theorem is a specialization of \cref{thm:phillips} when $k=2$ and $A_1 = -I, A_2 = M, B_1 = M, B_2 = I$. Straubing's proof of the Cayley--Hamilton theorem~\cite{straubing-1983} gives a weight-preserving and sign-reversing involution on $\mathcal{A}(b,e)$. Our proof when specialized to the Cayley--Hamilton theorem presents a weight-preserving and sign-reversing involution directly on $\mathcal{G}(b,e)$.

The constraint $A_1B_1 + A_2B_2 = 0$, in this case, is $(-I)M + M(I) = 0$ which means
\begin{equation}
\begin{split}
\label{eq:ch_specialize_1}
  \swgt\left( \taedge{x}{y}{1} \right) &= -\swgt\left( \tbedge{x}{y}{2} \right) = -\delta_{x,y}, \\
  \swgt\left( \taedge{x}{y}{2} \right) &= \swgt\left( \tbedge{x}{y}{1} \right) = M_{x,y}.
\end{split}
\end{equation}
Therefore, we also have
\begin{equation}
\begin{split}
\label{eq:ch_specialize_2}
  &\swgt\left( \tbedge{y}{z}{1} \right)
  \swgt \left( \taedge{z}{z}{1} \right) \\
  &= \swgt\left( \taedge{y}{y}{1} \right)
  \swgt \left( \tbedge{y}{z}{1} \right).
\end{split}
\end{equation}
Now consider the sum over $a$ in the left hand side of \eqref{eq:lemma1_second}. For example
\[
\wgt(\taedge{s}{a}{1}) \wgt( \tbedge{a}{q_2}{1}) = 
\delta_{s,a} M_{a,q_2},
\]
and therefore $a = s$. In that case 
\begin{align*}
&\swgt\left( \taedge{s}{s}{1} \right)
\swgt \left(  \tbedge{s}{q_2}{1} \right) \\
=& \swgt\left( \tbedge{s}{q_2}{1} \right)
\swgt \left( \taedge{q_2}{q_2}{1} \right)\\
=& -\swgt\left( \tbedge{s}{s}{2} \right)
\swgt \left( \taedge{s}{q_2}{2} \right)\\
=& -\swgt\left( \taedge{s}{q_2}{2} \right)
\swgt \left( \tbedge{q_2}{q_2}{2} \right),
\end{align*}
where the first equality follows by ~\eqref{eq:ch_specialize_2}, and the second and third by~\eqref{eq:ch_specialize_1}.
This shows that the two terms in \eqref{eq:lemma1_second} cancel pairwise for $\ell_s = 1,2$, and demonstrates the involution on $\mathcal{G}(b,e)$.

Notice that our proof strategy does not reduce to an involution on $\mathcal{A}(b,e)$. Therefore, we have a different combinatorial proof of the Cayley--Hamilton theorem as compared to the one by Straubing~\cite{straubing-1983}. 

\subsection{Relation to the proof by Phillips}
\label{sec:proof phillips}

We show now that our combinatorial proof is a reinterpretation of the algebraic proof of \cref{thm:phillips} by Phillips~\cite{phillips-1919}. Recall that we have matrices $A_1,\dots,A_k$, $B_1,\dots,B_k$ satisfying \eqref{aibisum}, where the $B_i$'s commute pairwise.
Let $M(x_1,\dots,x_k) = \left( A_{1}x_1 + \dots + A_{k}x_k \right)_{1 \le i,j \le n}$ be an $n \times n$ matrix 
and $M_{i,j}(x_1,\dots,x_k)$ be its $(i,j)$'th entry.
Then, let 
\begin{equation}
\label{MBij}
M^B_{i,j} = M_{i,j}(B_1, \dots, B_k) = (A_1)_{i,j} B_1 + \cdots + (A_k)_{i,j} B_k
\end{equation}
be the $n \times n$ matrix 
obtained by setting $B_i$ in place of $x_i$ for $i \in [k]$.
For a matrix $A$, let $A[i|j]$ be the
matrix $A$ with row $i$ and column $j$ removed, and denote $\det_B M[i|j]$ to be the matrix obtained by substituting $B_i$ in place of $x_i$ for $i \in [k]$ in $\det \big(M(x_1,\dots,x_k)[i|j] \big)$ so that
\begin{equation}
\label{detBM}
\sideset{}{_B}\det  M[i|j] = (-1)^{i+j} \sum_{\substack{\sigma \in S_n \\ \sigma_i = j}}
\sgn(\sigma) \prod_{\substack{r=1 \\ r \neq i}}^n M^B_{r,\sigma_r},
\end{equation}
using \eqref{MBij}.

Let us compute the signed weight of $\mathcal{G}(b,e)$, which we know
by \cref{lem:Gij} to be 0.
\begin{align*}
 \sum_{(\overline{m},\bar{q}) \in \mathcal{G}(b,e)} \swgt(\overline{m},\bar{q}) 
 = &\sum_{j=1}^n \sum_{(\bar{\pi}, \bar{q}) \in \mathcal{A}(b,e)} 
 \swgt(\phi((\bar{\pi}, \bar{q}), j ))\\
  =&\sum_{j=1}^n \sum_{\sigma \in S_n} \sgn(\sigma) 
  \sum_{(\bar{\sigma}, \bar{q}) \in \mathcal{A}(b,e)} 
  \wgt(\phi((\bar{\sigma}, \bar{q}), j))\\
  =& \sum_{j=1}^n \sum_{s=1}^n 
  \sum_{\substack{\sigma \in S_n\\ \sigma_s=b}} \sgn(\sigma)
  \sum_{\ell_1,\dots,\ell_n = 1}^k (A_{\ell_s})_{s,j} \\
  &\times \left( \prod_{\substack{r=1 \\ r \neq s}}^n  (A_{\ell_r})_{r,\sigma_r} \right)
  \left( \prod_{r=1}^n B_{\ell_r} \right)_{j,e}.
\end{align*}
Next we rely on the commutativity of the $B_i$'s to write this as
\begin{align*}
  & \sum_{j=1}^n \sum_{s=1}^n 
  \sum_{\substack{\sigma \in S_n\\ \sigma_s=b}} \sgn(\sigma)
  \sum_{\ell_1,\dots,\ell_n = 1}^k \Big( (A_{\ell_s})_{s,j}B_{\ell_s} \prod_{\substack{r=1 \\ r \neq s}}^n(A_{\ell_r})_{r,\sigma_r} B_{\ell_r} \Big)_{j,e} \\
  =& \sum_{j=1}^n \sum_{s=1}^n 
  \sum_{\substack{\sigma \in S_n\\ \sigma_s=b}} \sgn(\sigma)
  \left( M^B_{s,j} \prod_{\substack{r=1 \\ r \neq s}}^n M^B_{r,\sigma_r} \right)_{j,e},
\end{align*}
where we have first performed the $\ell_1,\dots,\ell_n$ sums before taking the $(j,e)$'th entry and used \eqref{MBij} in the last step.
Now let us perform the inner sum.
Since the product over $r \ne s$ is not dependent on $j$, we can use \eqref{detBM} to arrive at
\[
\sum_{j=1}^n \sum_{s=1}^n (-1)^{s+j} \left( M^B_{s,j} \sideset{}{_B}\det  M[s|b]   \right)_{j,e}.
\]
By the standard Laplace expansion, the only contribution to the $j$ sum comes from $j = b$, giving
\[
\sum_{s=1}^n (-1)^{s+b} \left( M^B_{s,b} \sideset{}{_B}\det  M[s|b] \right)_{b,e}.
\]
This is precisely what Phillips shows to be 0 in \cite[Theorem I]{phillips-1919}.

\section{Application to Mixed Discriminants}
\label{sec:mixeddisc}

In this section, we will prove \cref{thm:mixeddisc} using the same strategy as for the proof of \cref{thm:phillips} in \cref{sec:phillips}.
We recall the setup. We have $2nk$ matrices, which we call $(A_{i,j})_{i \in [n], j \in [k]}$ and $(B_{i,j})_{i \in [n], j \in [k]}$ which satisfy the conditions:
\begin{itemize}
\item $B_{i,j}B_{i',j'} = B_{i',j'}B_{i,j} \text{ for all } 1 \le i < i' \le n, 1 \le j, j' \le k$, and

\item
\begin{equation}
\label{aijbijsum}
A_{i,1}B_{i,1} + \cdots + A_{i,k}B_{i,k} = 0, \quad 1 \le i \le n.
\end{equation}

\end{itemize}

Recall the polynomial $\hat{p} \left((x_{i,j})_{i \in [n], j \in [k]} \right)$ from \eqref{defq},
\[
\hat{p} \left((x_{i,j})_{i \in [n], j \in [k]} \right) = \mathfrak{D}(A_{1,1}x_{1,1} + \dots + A_{1,k}x_{1,k}, \dots, A_{n,1}x_{n,1} + \dots + A_{n,k}x_{n,k}),
\]
in $nk$ variables $(x_{i,j})_{i \in [n], j \in [k]}$, where $\mathfrak{D}$ is the mixed discriminant given in \cref{def:mixeddisc}.

\begin{defn}
A \emph{decorated $2$-permutation} $\hat{\alpha}_\pi$ is an $n$-tuple of quadruples $(\hat{\alpha}_\pi)_i = (i, \alpha_i, \pi_i, \ell_i)$ for $i \in [n], \ell_i \in [k]$, where $\pi, \alpha \in S_n$ and the pairs $\alpha_{\pi_i}, \ell_i$ are called \emph{labels}. 
We will denote $(\hat{\alpha}_\pi)_i$ as \taedge{i}{\pi_i}{\alpha_{\pi_i},\ell_i}, which has \emph{weight} $(A_{\alpha_{\pi_i},\ell_i})_{i,\pi_i}$.
The \emph{signed weight} of the decorated $2$-permutation is given by
\[
\swgt(\hat{\alpha}_\pi) = \sgn(\pi) \prod_{i=1}^n (A_{\alpha_{\pi_i}, \ell_i})_{i,\pi_i}.
\]
The set of all decorated $2$-permutations is denoted $\hat{S}^2_{n,k}$.
\end{defn}

Since there are $n!$ permutations and all labels are independently chosen, the cardinality of $\hat{S}^2_{n,k}$ is $n!^2 k^n$.

\begin{defn}
A \emph{decorated $2$-path} of length $n$ is a tuple $\hat{q} = (q_1,\dots,q_{n+1})$, where each $q_i \in [n]$. 
For $i \in [n]$, the $i$'th labeled edge is denoted $\hat{q}_i = \tbedge{q_i}{ q_{i+1}}{\alpha_i,  \ell_i}$
and has weight $(B_{\alpha_i, \ell_i})_{q_i,q_{i+1}}$, where $\alpha \in S_n$ and the \emph{label} $\ell_i \in [k]$.
The \emph{weight} of the decorated $2$-path is
\[
\wgt(\hat{q}) = \prod_{i=1}^n (B_{\alpha_{i}, \ell_i})_{q_i, q_{i+1}}.
\]
The set of all decorated $2$-paths is denoted $\widehat{Q}^2_{n,k}$.
\end{defn}

\begin{defn}
A \emph{$2$-pathmutation} is a pair $(\hat{\alpha}_\pi, \hat{q})$ where $\hat{\alpha}_\pi \in \hat{S}^2_{n,k}, \hat{q} \in \widehat{Q}^2_{n,k}$ such that the labels of the $i$'th element of the permutation and the $i$'th edge of the path are the same for all $i \in [n]$.
The \emph{signed weight} of a $2$-pathmutation is 
\[
\wgt(\hat{\alpha}_\pi, \hat{q}) = \swgt(\hat{\alpha}_\pi) \wgt(\hat{q}).
\]
The set of $2$-pathmutations beginning with $q_1 = b$ and ending with $q_{n+1} = e$ is denoted $\mathcal{A}^2(b,e)$.
\end{defn}

The cardinality of $\mathcal{A}^2(b,e)$ is $n!^2 k^n n^{n-1}$ for every $b,e \in [n]$ because we can choose $q_2,\dots,q_{n-1}$ arbitrarily.

\cref{fig:discriminant_PQ} shows a $2$-pathmutation $(\hat{\alpha}_\pi, \hat{q})$.

\begin{center}
  \begin{figure}[h!]
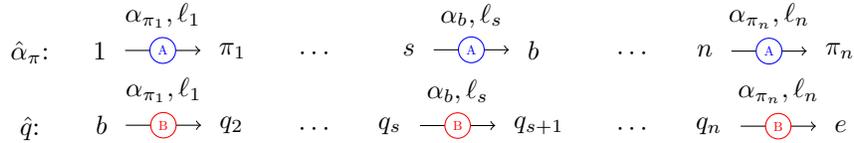

    \begin{tabular}{c c c c c c}
      $\hat{\alpha}_\pi$: & \taedge{1}{\pi_1}{ \alpha_{\pi_1},\ell_1} & $\cdots$ & \taedge{s}{b}{\alpha_{b},\ell_s} & $\cdots$ & \taedge{n}{ \pi_n}{\alpha_{\pi_n},\ell_n}\\
      $\hat{q}$: & \tbedge{b}{q_2}{\alpha_{\pi_1},\ell_1} & $\cdots$ & \tbedge{q_s}{ q_{s+1}}{\alpha_{b},\ell_s} & $\cdots$ & \tbedge{q_n}{e}{\alpha_{\pi_n},\ell_n}
    \end{tabular}
    \caption{A $2$-pathmutation $(\hat{\alpha}_\pi, \hat{q}) \in \mathcal{A}^2(b,e)$ where $\pi_s = b$.}
    \label{fig:discriminant_PQ}
  \end{figure}
\end{center}

\begin{defn}
A \emph{decorated $2$-map} $\widehat{m}$ is an $n$-tuple of quadruples
$\widehat{m}_i = (\sigma_i, \tau_i, \alpha_i, \ell_i)$, where $\sigma \in S_n$ is either the identity or a single transposition, $\alpha \in S_n$, $\tau_i \in [n]$ and $\ell_i \in [k]$ for all $i$ such that
\begin{itemize}
\item $\# \{\tau_1,\dots, \tau_n \} \geq n-1$,
\item if $\# \{\tau_1,\dots, \tau_n \} = n$, then $\sigma$ is the identity permutation,
\item if $\tau_i = \tau_j$ for some $(i,j)$, then either 
$\sigma_i = i, \sigma_j = j$ or $\sigma_i = j, \sigma_j = i$.
\end{itemize}
The \emph{weight} of $\widehat{m}_i$ is $(A_{\alpha_i,\ell_i})_{\sigma_i,\tau_i}$ and is denoted \taedge{\sigma_i}{\tau_i}{\alpha_i,\ell_i}. The weight of $\widehat{m}$ is then
\[
\wgt(\widehat{m}) = \prod_{i=1}^n (A_{\alpha_i,\ell_i})_{\sigma_i,\tau_i}.
\]
The set of all decorated $2$-maps is denoted $\widehat{M}^2_{n,k}$.
\end{defn}

When $\# \{\tau_1,\dots, \tau_n \} = n$, we get exactly decorated $2$-permutations. When $\# \{\tau_1,\dots, \tau_n \} = n-1$, there are $n(n-1) \times n!/2$ possibilities for $\tau$ and $2$ possibilities for $\sigma$ so that the cardinality of $\widehat{M}^2_{n,k}$ is
\[
k^n (n! + n(n-1) n!) = n!^2 k^n (n^2 - n + 1).
\]

\begin{defn}
A \emph{$2$-pathmap} is a pair $(\widehat{m},\hat{q})$ where $\widehat{m} \in \widehat{M}^2_{n,k}, \hat{q} \in \widehat{Q}^2_{n,k}$ such that 

\begin{itemize}

\item If $\{\tau_1,\dots,\tau_n\} = [n]$, then $(\widehat{m},\hat{q})$ is a $2$-pathmutation.

\item If $\# \{\tau_1,\dots, \tau_n \} = n-1$ and $\tau_i = \tau_j$ for some $i \neq j$, then $q_1 = \tau_i$. In this case, the labels of $\widehat{m}_k$ and $\hat{q}_k$ must match for all $k \neq i,j$. In addition, if $\sigma_i = i, \sigma_j = j$ (resp. $\sigma_i = j, \sigma_j = i$), then the labels of $\widehat{m}_i$ and $\widehat{m}_j$ are equal to those of $\hat{q}_i, \hat{q}_j$ (resp. $\hat{q}_j, \hat{q}_i$) respectively.

\end{itemize}

The weight of the $2$-pathmap $(\widehat{m},\hat{q})$ is
\[
\wgt(\widehat{m},\hat{q}) = \wgt(\widehat{m}) \wgt(\hat{q}).
\]

The set of $2$-pathmaps with $\# \{\tau_1,\dots, \tau_n \} = n-1$ such that $\{\tau_1,\dots, \tau_n \} = [n] \setminus \{b\}$ and ending with $q_{n+1} = e$ is denoted $\mathcal{H}^2(b,e)$. 
In addition, let $\mathcal{G}^2(b,e) = \mathcal{H}^2(b,e) \cup \mathcal{A}^2(b,e)$.
\end{defn}

Analogous to the enumeration of pathmaps, the cardinality of $\mathcal{H}^2(b,e)$ is again $n-1$ times that of $\mathcal{A}^2(b,e)$.
As in the proof of Phillips' theorem, we will need to attach a sign to a $2$-pathmap in $\mathcal{G}^2(b,e)$.
As before, define a map $\hat\phi: \mathcal{A}^2(b,e) \times [n] \to \mathcal{G}^2(b,e)$.
Set $\hat\phi((\hat{\alpha}_\pi, \hat{q}), j) = (\widehat{m}',\hat{q}')$  as follows. 
First, set 
\[
\hat{q}'_r = \begin{cases}
\tbedge{i'}{q_2}{\alpha_{\pi_1},\ell_1} & r = 1, \\
\hat{q}_r & \text{otherwise}.
\end{cases}
\]
Next, set $s = \pi^{-1}_b$ and $t = \pi^{-1}_j$.
Then let
\begin{equation}
\label{defphihat}
\widehat{m}'_r = 
\begin{cases}
\taedge{{\min(s,t)}}{j}{\alpha_{\pi_{\min(s,t)}},\ell_{\min(s,t)}} & \text{if $r = s$}, \\
\taedge{{\max(s,t)}}{j}{\alpha_{\pi_{\max(s,t)}},\ell_{\max(s,t)}} & \text{if $r = t$},\\
(\hat{\alpha}_\pi)_r & \text{otherwise}.
\end{cases}
\end{equation}

The sign of an element $(\widehat{m},\hat{q}) \in \mathcal{G}^2(b,e)$ can be defined in the same way as before. If $\hat\phi^{-1}(\widehat{m},\hat{q}) = ((\hat{\alpha}_\pi, \hat{q}), j)$, then
the signed weight of $(\widehat{m},\hat{q})$ is given by
\begin{equation}
\label{sgn-bipathmap}
\swgt(\widehat{m},\hat{q}) = \sgn(\pi) \wgt(\widehat{m},\hat{q}).
\end{equation}

\begin{proof}[Proof of \cref{thm:mixeddisc}]

We first claim that
\begin{equation}
\label{A2-sum}
\sum_{(\hat{\alpha}_\pi,\hat{q}) \in \mathcal{A}^2(b,e)} \swgt(\hat{\alpha}_\pi,\hat{q}) = \hat{p}\left((B_{i,j})_{i \in [n], j \in [k]} \right)_{b,e}.
\end{equation}
To see this, begin by expanding the polynomial $q$ in \eqref{defq}  as
\begin{multline*}
 \hat{p} \left((x_{i,j})_{i \in [n], j \in [k]} \right) = \frac{1}{n!}\sum_{\alpha \in S_n} \sum_{\pi \in S_n} \sgn(\pi) \\
 \times \prod_{i=1}^n \Big(
(A_{\alpha_{\pi_i}, 1})_{i, \pi_i} x_{\alpha_{\pi_i}, 1} + \cdots +
(A_{\alpha_{\pi_i}, k})_{i, \pi_i} x_{\alpha_{\pi_i}, k}
\Big).
\end{multline*}
Now, substitute $x_{i,j}$ by $B_{i,j}$ and use the fact that $B_{i,j}$'s commute pairwise to obtain
\begin{multline*}
\hat{p} \left((B_{i,j})_{i \in [n], j \in [k]} \right) =
\frac{1}{n!}\sum_{\alpha \in S_n} \sum_{\pi \in S_n} \sgn(\pi) \\
\times \prod_{i=1}^n \Big(
(A_{\alpha_{\pi_i}, 1})_{i, \pi_i} B_{\alpha_{\pi_i}, 1} + \cdots +
(A_{\alpha_{\pi_i}, k})_{i, \pi_i} B_{\alpha_{\pi_i}, k}
\Big),
\end{multline*}
which now simplifies to
\[
\hat{p} \left((B_{i,j})_{i \in [n], j \in [k]} \right) = \frac{1}{n!}\sum_{\alpha \in S_n} \sum_{\pi \in S_n} \sgn(\pi) \sum_{(\ell_1, \dots, \ell_n) \in [k]^n} \prod_{i=1}^n (A_{\alpha_{\pi_i}, \ell_i})_{i, \pi_i}B_{\alpha_{\pi_i}, \ell_i}.
\]
Now consider the $(b,e)$'th entry of this sum.
For each pair of permutations $\alpha,\pi$ and each element $\ell = (\ell_1, \dots, \ell_n) \in [k]^n$, we can represent the product of $A_{\alpha_{\pi_i}, \ell_i}$ over $i$ as the weight of the decorated $2$-permutation $\hat{\alpha}_\pi$, the label of whose $i$'th element is $(\alpha_{\pi_i},\ell_i)$ as seen above.
Now expand the product of $B_{\alpha_{\pi_i},\ell_i}$'s on the right hand side.
The $(b,e)$'th entry is a sum of terms, each of which corresponds exactly to a decorated $2$-path with initial vertex $b$ and final vertex $e$.
The $i$'th edge in the decorated $2$-path has the same label, $(\alpha_{\pi_i},\ell_i)$. Therefore, each term corresponds to a $2$-pathmutation, whose weight is equal to the term. This proves the claim above.

We now have the analogues of \cref{lem:Gij} and \cref{lem:Hij}. 
\begin{align}
\label{G2ij}
\sum_{(\widehat{m},\hat{q}) \in \mathcal{G}^2(b,e)} \swgt(\widehat{m},\hat{q}) =& 0, \\
\label{H2ij}
\sum_{(\widehat{m},\hat{q}) \in \mathcal{H}^2(b,e)} \swgt(\widehat{m},\hat{q}) =& 0.
\end{align}
The proofs of these equations proceed in essentially the same manner as the above lemmas and we omit them. By definition, the left hand side of \eqref{A2-sum} is the difference of the left hand sides of \eqref{G2ij} and \eqref{H2ij}, proving the result.
\end{proof}

\subsection{Illustration for $n=2$}
\label{sec:mixed-n=2}

We illustrate the ideas in the proof of \eqref{G2ij} and \eqref{H2ij}, which are key to the proof of \cref{thm:mixeddisc}, for $n = k = 2$. As in \cref{sec:n=2}, we will look at $b = 1$ and $e = 2$ in detail and keep the labels $\ell_1 = r$ and $\ell_2 = s$ in addition to the permutation $\alpha$ arbitrary. We then have $4$ $2$-pathmutations, which are shown in the left columns of \cref{fig:gen-eg-n=2ab,fig:gen-eg-n=2cd,fig:gen-eg-n=2efgh}. 
Similarly, there are $4$ such $2$-pathmaps in $\mathcal{H}^2(1,2)$, which are shown in the right columns of \cref{fig:gen-eg-n=2ab,fig:gen-eg-n=2cd,fig:gen-eg-n=2efgh}.

\begin{center}
\begin{figure}[h!]
\begin{tabular}{c c}
      \resizebox{0.45\textwidth}{!}{
    \begin{tabular}{c c c}
        $\hat{\alpha}_{(1,2)}$: & \taedge{1}{1}{\alpha_1,r} & \taedge{2}{2}{\alpha_2,s}\\[0.2cm]
        $\hat{q}$: & \tbedge{1}{1}{\alpha_1,r} & \tbedge{1}{2}{\alpha_2,s}
    \end{tabular}}
&
      \resizebox{0.45\textwidth}{!}{
    \begin{tabular}{c c c}
        $\widehat{m}$: & \taedge{1}{{2}}{\alpha_1,r} & \taedge{2}{2}{\alpha_2,s}\\[0.2cm]
        $\hat{q}$: & \tbedge{{2}}{1}{\alpha_1,r} & \tbedge{1}{2}{\alpha_2,s}
    \end{tabular}}
\\
    (a) {\tiny $+(A_{\alpha_1,r})_{1,1}(A_{\alpha_2,s})_{2,2}(B_{\alpha_1,r})_{1,1}(B_{\alpha_2,s})_{1,2}$}
&
    (b) {\tiny $+(A_{\alpha_1,r})_{1,2}(A_{\alpha_2,s})_{2,2}(B_{\alpha_1,r})_{2,1}(B_{\alpha_2,s})_{1,2}$}
\end{tabular}
\caption{The $2$-pathmap terms proportional to $(A_{\alpha_2,s})_{2,2}(B_{\alpha_2,s})_{1,2}$ along with their signed weights.}
\label{fig:gen-eg-n=2ab}
\end{figure}
\end{center}

\begin{center}
\begin{figure}[h!]
\begin{tabular}{c c}
      \resizebox{0.45\textwidth}{!}{
    \begin{tabular}{c c c}
        $\hat{\alpha}_{(1,2)}$: & \taedge{1}{1}{\alpha_1,r} & \taedge{2}{2}{\alpha_2,s}\\[0.2cm]
        $\hat{q}$: & \tbedge{1}{2}{\alpha_1,r} & \tbedge{2}{2}{\alpha_2,s}
    \end{tabular}}
&
      \resizebox{0.45\textwidth}{!}{
    \begin{tabular}{c c c}
        $\widehat{m}$: & \taedge{1}{{2}}{\alpha_1,r} & \taedge{2}{2}{\alpha_2,s}\\[0.2cm]
        $\hat{q}$: & \tbedge{{2}}{2}{\alpha_1,r} & \tbedge{2}{2}{\alpha_2,s}
    \end{tabular}}
\\
    (c) {\tiny $+(A_{\alpha_1,r})_{1,1}(A_{\alpha_2,s})_{2,2}(B_{\alpha_1,r})_{1,2}(B_{\alpha_2,s})_{2,2}$}
&
    (d) {\tiny $+(A_{\alpha_1,r})_{1,2}(A_{\alpha_2,s})_{2,2}(B_{\alpha_1,r})_{2,2}(B_{\alpha_2,s})_{2,2}$}
\end{tabular}
\caption{The $2$-pathmap terms proportional to $(A_{\alpha_2,s})_{2,2}(B_{\alpha_2,s})_{2,2}$ along with their signed weights.}
\label{fig:gen-eg-n=2cd}
\end{figure}
\end{center}

Let us compare the $2$-pathmutation in \cref{fig:gen-eg-n=2ab}(a) and the $2$-pathmap in \cref{fig:gen-eg-n=2ab}(b).
One can check that the latter has positive sign.
Now, the sum of weights of these are
\begin{align*}
& \sum_{\alpha \in S_2} \sum_{r=1}^2 \sum_{s=1}^2 (A_{\alpha_2,s})_{2,2}(B_{\alpha_2,s})_{1,2} \Big(
(A_{\alpha_1,r})_{1,1}(B_{\alpha_1,r})_{1,1} + (A_{\alpha_1,r})_{1,2}(B_{\alpha_1,r})_{2,1} \Big) \\
=& \sum_{\alpha_1 = 1}^2 \sum_{s=1}^2 (A_{3-\alpha_1,s})_{2,2}(B_{3-\alpha_1,s})_{1,2} 
\sum_{r=1}^2 (A_{\alpha_1,r} B_{\alpha_1,r})_{1,1}.
\end{align*}
which is zero by \eqref{aijbijsum}. 
A very similar computation goes through for 
the terms in \cref{fig:gen-eg-n=2cd}(c) and (d) and gives
\begin{align*}
\sum_{\alpha_1 = 1}^2 \sum_{s=1}^2 (A_{3-\alpha_1,s})_{2,2}(B_{3-\alpha_1,s})_{2,2} 
\sum_{r=1}^2 (A_{\alpha_1,r} B_{\alpha_1,r})_{1,2},
\end{align*}
which is also zero for the same reason.

\begin{center}
\begin{figure}[h!]
\begin{tabular}{c c}
      \resizebox{0.45\textwidth}{!}{
    \begin{tabular}{c c c}
        $\hat{\alpha}_{(2,1)}$: & \taedge{1}{2}{\alpha_2,r} & \taedge{2}{1}{\alpha_1,s}\\[0.2cm]
        $\hat{q}$: & \tbedge{1}{1}{\alpha_2,r} & \tbedge{1}{2}{\alpha_1,s}
    \end{tabular}}
&
      \resizebox{0.45\textwidth}{!}{
    \begin{tabular}{c c c}
        $\widehat{m}$: & \taedge{2}{{2}}{\alpha_1,s} & \taedge{1}{2}{\alpha_2,r}\\[0.2cm]
        $\hat{q}$: & \tbedge{{2}}{1}{\alpha_2,r} & \tbedge{1}{2}{\alpha_1,s}
    \end{tabular}}
\\
    (e) {\tiny $-(A_{\alpha_2,r})_{1,2}(A_{\alpha_1,s})_{2,1}(B_{\alpha_2,r})_{1,1}(B_{\alpha_1,s})_{1,2}$}
&
    (f) {\tiny $-(A_{\alpha_2,r})_{1,2}(A_{\alpha_1,s})_{2,2}(B_{\alpha_2,r})_{2,1}(B_{\alpha_1,s})_{1,2}$}
\\[0.4cm]
      \resizebox{0.45\textwidth}{!}{
    \begin{tabular}{c c c}
        $\hat{\alpha}_{(1,2)}$: & \taedge{1}{2}{\alpha_2,r} & \taedge{2}{1}{\alpha_1,s}\\[0.2cm]
        $\hat{q}$: & \tbedge{1}{2}{\alpha_2,r} & \tbedge{2}{2}{\alpha_1,s}
    \end{tabular}}
&
      \resizebox{0.45\textwidth}{!}{
    \begin{tabular}{c c c}
        $\widehat{m}$: & \taedge{2}{{2}}{\alpha_1,s} & \taedge{1}{2}{\alpha_2,r}\\[0.2cm]
        $\hat{q}$: & \tbedge{{2}}{2}{\alpha_2,r} & \tbedge{2}{2}{\alpha_1,s}
    \end{tabular}}
\\
    (g) {\tiny $-(A_{\alpha_2,r})_{1,2}(A_{\alpha_1,s})_{2,1}(B_{\alpha_2,r})_{1,2}(B_{\alpha_1,s})_{2,2}$}
&
    (h) {\tiny $-(A_{\alpha_2,r})_{1,2}(A_{\alpha_1,s})_{2,2}(B_{\alpha_2,r})_{2,2}(B_{\alpha_1,s})_{2,2}$}
\end{tabular}

  \caption{The $2$-pathmap terms proportional to $(A_{\alpha_2,r})_{1,2}$ along with their signed weights.}
  \label{fig:gen-eg-n=2efgh}
\end{figure}
\end{center}

Complications arise in the remaining terms shown in \cref{fig:gen-eg-n=2efgh}(e), (f), (g) and (h). The sign for the terms in (g) and (h) are computed as described above. In this case, combining terms (e) and (g), we get
\begin{align*}
& -\sum_{\alpha \in S_2} \sum_{r=1}^2 \sum_{s=1}^2 (A_{\alpha_2,r})_{1,2}(A_{\alpha_1,s})_{2,1} \Big(
(B_{\alpha_2,r})_{1,1}(B_{\alpha_1,s})_{1,2} + (B_{\alpha_2,r})_{1,2}(B_{\alpha_1,s})_{2,2} \Big) \\
=& -\sum_{\alpha \in S_2} \sum_{r=1}^2 \sum_{s=1}^2 (A_{\alpha_2,r})_{1,2}(A_{\alpha_1,s})_{2,1} \sum_{r=1}^2 (B_{\alpha_2,r} B_{\alpha_1,s})_{1,2} \\
=& -\sum_{\alpha \in S_2} \sum_{r=1}^2 \sum_{s=1}^2 (A_{\alpha_2,r})_{1,2}(A_{\alpha_1,s})_{2,1} \sum_{r=1}^2 (B_{\alpha_1,s} B_{\alpha_2,r})_{1,2} \\
=& -\sum_{\alpha \in S_2} \sum_{r=1}^2 \sum_{s=1}^2 (A_{\alpha_2,r})_{1,2}(A_{\alpha_1,s})_{2,1} \Big(
(B_{\alpha_1,s})_{1,1}(B_{\alpha_2,r})_{1,2} + (B_{\alpha_1,s})_{1,2}(B_{\alpha_2,r})_{2,2} \Big),
\end{align*}
where we have used the commutativity of $B_{\alpha_2,r}$ and $B_{\alpha_1,s}$ in the third line. Similarly, combining terms (f) and (h), we get
\begin{align*}
\sum_{\alpha \in S_2} \sum_{r=1}^2 \sum_{s=1}^2 (A_{\alpha_2,r})_{1,2}(A_{\alpha_1,s})_{2,2} \Big(
(B_{\alpha_2,r})_{2,1}(B_\alpha)_{1,2} + (B_{\alpha_1,s})_{2,2}(B_{\alpha_2,r})_{2,2} \Big).
\end{align*}
Now, add the first summands in both the above equations to obtain
\begin{align*}
&  -\sum_{\alpha \in S_2} \sum_{r=1}^2 \sum_{s=1}^2 (A_{\alpha_2,r})_{1,2} (B_{\alpha_2,r})_{1,2} \Big(
(A_{\alpha_1,s})_{2,1} (B_{\alpha_1,s})_{1,1} + (A_{\alpha_1,s})_{2,2}(B_\beta)_{2,1} \Big) \\
=& -\sum_{\alpha_1 = 1}^2 \sum_{r=1}^2(A_{3-\alpha_1,r})_{1,2} (B_{3-\alpha_1,r})_{1,2} \sum_{s=1}^2 (A_{\alpha_1,s} B_{\alpha_1,s})_{2,1},
\end{align*}
which is now 0 by \eqref{aijbijsum}. A similar computation goes through for the second and fourth summands. 
This kind of computation is what needs to carried out to prove \eqref{G2ij}.

Now focus on the $2$-pathmap terms in \cref{fig:gen-eg-n=2ab,fig:gen-eg-n=2cd,fig:gen-eg-n=2efgh}, namely (b), (d), (f) and (h). 
Focus on the (b) and (f) figures. If we interchange $\alpha_1$ and $\alpha_2$ in the weight of the (b) figure, we obtain the negative of the weight of the (f) figure. Similarly, for (d) and (h) terms. Thus, an involution of the same kind used in the proof of \cref{lem:Hij} in addition to an appropriate involution on $\alpha$ will prove \eqref{H2ij}.

\subsection*{Acknowledgments} 
We thank R. B. Bapat for many helpful discussions.
We acknowledge support from the UGC Centre for Advanced Studies and from SERB grant CRG/2021/001592.

\bibliography{cayley_hamilton}
\bibliographystyle{alpha}

\end{document}